\documentclass[10pt]{amsart}
\usepackage{amsmath, amsthm, amssymb,mathscinet}
\usepackage{graphicx}
\usepackage{mathbbol}
\usepackage{txfonts}
\usepackage[usenames]{color}
\usepackage{graphicx}

\newtheorem{thm}{Theorem}[section]
\newtheorem{ex}[thm]{Example}

\newtheorem{lem}[thm]{Lemma}
\newtheorem{prop}[thm]{Proposition}
\newtheorem{defn}[thm]{Definition}
\newtheorem{rmrk}[thm]{Remark}

\newtheorem*{mthm*}{Main Theorem}

\newtheorem*{thmA'}{Theorem A'}
\newtheorem*{conj*}{Conjecture}
\numberwithin{equation}{section}

\newcommand{\be}{\begin{equation}}
\newcommand{\ee}{\end{equation}}
\newcommand{\bee}{\begin{equation*}}
\newcommand{\eee}{\end{equation*}}

\newcommand{\N}{\mathbb{N}}

\newcommand{\R}{\mathbb{R}}
\newcommand{\One}{{\bf \rm{1}}}

\newcommand{\Z}{\mathbb{Z}}

\newcommand{\union}{\cup}

\newcommand{\AlexnkD}{\text{Alex}^n(\kappa,D)}

\newcommand{\diam}{\rm{diam}}
\newcommand{\Alex}{\text{Alex\,}}

\newcommand{\Alexnk}{\text{Alex}^n(\kappa)}

\newcommand{\tang}[4]{\tilde\measuredangle_{#1}({#2}\,_{#4}^{#3})}

\newcommand{\GHto}{\stackrel { \textrm{GH}}{\longrightarrow} }
\newcommand{\Fto}{\stackrel {\mathcal{F}}{\longrightarrow} }

\newcommand{\Fm}{{\mathcal F}}

\newcommand{\set}{\rm{set}}
\newcommand{\Lip}{\operatorname{Lip}}

\newcommand{\mass}{{\mathbf M}}
\newcommand{\curr}{{\mathbf M}}         
\newcommand{\intrectcurr}{{\mathcal I}} 
\newcommand{\intcurr}{{\mathbf I}}      

\newcommand{\fillvol}{{\operatorname{FillVol}}}

\newcommand{\rstr}{\:\mbox{\rule{0.1ex}{1.2ex}\rule{1.1ex}{0.1ex}}\:}
\newcommand{\bdry}{\partial}
\newcommand{\spt}{\operatorname{spt}}

\begin{document}

\title{On the Sormani-Wenger Intrinsic Flat Convergence of Alexandrov Spaces}

\author{Nan Li}
\address{Department of Mathematics, The City University of New York, 300 Jay Street, Brooklyn, NY, 11201}
\email{lilinanan@gmail.com, NLi@citytech.cuny.edu}

\author{Raquel Perales}
\address{Conacyt Research Fellow, Instituto de Matematicas, UNAM, Oaxaca}
\email{raquel.peralesaguilar@gmail.com}
\thanks{The first author was partially supported by the research funds managed by Penn State University. The second author's research was funded in part by Prof. Christina Sormani's NSF Grant DMS 10060059. In addition, the second author received funding from Stony Brook as a doctoral student.}
\newcommand{\rp}[1]{{\textcolor{blue}{#1}}}

\date{\today}

\begin{abstract}
We study sequences of  integral current spaces $(X_j,d_j,T_j)$ such that the integral current structure $T_j$ has weight $1$ and no boundary and, all $(X_j,d_j)$ are closed Alexandrov spaces with curvature uniformly bounded from below and diameter uniformly bounded from above.  We prove that for such sequences either their limits collapse or the Gromov-Hausdorff and Sormani-Wenger Intrinsic Flat limits agree.  The latter is done showing that the lower $n$ dimensional density of the mass measure at any regular point of the Gromov-Hausdorff limit space is positive by passing to a filling volume estimate. 

In an appendix we show that the filling volume of the standard $n$ dimensional integral current space coming from an  $n$ dimensional sphere of radius $r>0$ in Euclidean space  equals $r^n$ times the filling volume of the $n$ dimensional integral current space coming from the $n$ dimensional sphere of radius $1$. 

\end{abstract}
\maketitle{}


\section*{Introduction}

Burago, Gromov and Perelman proved that 
sequences of Alexandrov spaces with curvature 
uniformly bounded from below, diameter and dimension uniformly bounded from above, have subsequences which converge
in the Gromov-Hausdorff sense to an Alexandrov space with
the same curvature and diameter bounds.  The properties of Alexandrov spaces and their Gromov-Hausdorff limit spaces have been amply studied by Alexander-Bishop \cite{AlexanderBishop}, Alexander-Kapovitch-Petrunin  \cite{AKP}, Burago-Gromov-Perelman \cite{BGP}, Burago-Burago-Ivanov \cite{BBI}, Li-Rong \cite{LR10}, Otsu-Shioya  \cite{OS}, 
and many others.

The Gromov Hausdorff distance between two metric spaces, $X_i$, is defined as,
\be
d_{GH}\left(X_1,X_2\right) = \inf  \{ d^Z_H\left(\varphi_1\left(X_1\right), \varphi_2\left(X_2\right)\right) \},
\ee
where $d^Z_H$ denotes the Hausdorff distance and the infimum is taken over all complete metric spaces $Z$  and all distance preserving maps $\varphi_i: X_i \to Z$, \cite{Gro81}. 
Sormani-Wenger's intrinsic flat distance between integral current spaces is defined in imitation of Gromov-Hausdorff distance:
it replaces the Hausdorff distance $d_H$ by Federer-Fleming's flat distance $d_F$ \cite{CS11}.  The notion of  flat convergence  was first introduced in work by Whitney \cite{Whit} and extended to integral currents on Euclidean space by
Federer-Fleming \cite{FF}.  Then the notion of  integral current on Euclidean space was extended by Ambrosio-Kirchheim to integral currents on
complete metric spaces \cite{AK}.   Sormani-Wenger proved that sequences of $n$ dimensional integral current spaces  that are equibounded and have diameter and total mass uniformly bounded from above have subsequences that converge in both Gromov-Hausdorff and intrinsic flat sense. Either the intrinsic flat limit is contained in the Gromov-Hausdorff limit or the intrinsic flat limit is the zero integral current space \cite[Theorem  3.20]{CS11}. For example, a sequence of collapsing tori converges in Gromov-Hausdorff sense to a circle. Since the Hausdorff dimension of the circle does not coincide with the Hausdorff dimension of the tori, the intrinsic flat limit is the zero integral current space.  

The intrinsic flat and Gromov-Hausdorff limits have been proved to agree under certain special conditions.
Sormani-Wenger proved that  the Gromov-Hausdorff  and intrinsic flat limits agree for compact oriented sequences of $n$ dimensional Riemannian manifolds that have nonnegative Ricci curvature,  volume uniformly bounded from below and from above by positive constants and a uniform diameter bound \cite[Theorem 7.1]{CS10}.  The second author showed a similar result for manifolds with boundary  \cite{Perales-2}.  Matveev-Portegies provided a new proof of Sormani-Wenger's aforementioned result that allows for any lower Ricci curvature bound \cite{MP}.  In a more general setting and with different techniques, Honda obtains as a corollary Matveev-Portegies result \cite{Ho}.  Portegies-Sormani proved a Tetrahedral Compactness Theorem which shows that both limits agree for sequences of spaces that satisfy uniform tetrahedral property bounds \cite{PorSor}. The second author with NunezZimbron extended the notion of tetrahedral property and proved a Generalized Tetrahedral Compactness Theorem \cite{NuZiPe}.

The intrisic flat distance has also been studied by Jaramillo et al \cite{Ja}, Lakzian \cite{LakzianDiam}, Munn \cite{MunnInt}, Portegies \cite{PortegiesSemi} and the author \cite{Perales-1}, to mention a few.  Gromov proposes to use intrinsic flat distance 
to solve some of his conjectures \cite{Gro13}, \cite{Gro14}.

Before stating our theorem we recall that  Sormani-Wenger \cite{CS11}  defined
integer rectifiable current spaces,  $(Y,d,T)$, which are oriented countably $\mathcal{H}^n$ rectifiable metric spaces  $(Y,d)$ 
with an $n$ dimensional integer current structure $T$ on $\bar Y$ that satisfies the condition:
$
Y=\{y\in \bar{Y}:\, \liminf_{r\to 0} ||T||(B_r(y))/r^n>0\},
$
where $||T||$ is the mass measure of $T$.  Note that an integer rectifiable current, $T$,
is determined by the orientation and chart structure of $(Y,d)$ and an 
integer valued Borel weight function which in turn determines the mass measure $||T||$.  An integral current spaces $(Y,d,T)$ is an integer rectifiable current space whose boundary $\bdry T$ is an integer rectifiable current.
In this paper we assume that the Borel weight function determining the
integral current structure is $1$, in the sense of Definition \ref{def-weight}, and that there is no boundary, $\bdry T=0$.

\begin{thm}\label{thm-mainthm}
 Let $(X_j,d_j,T_j)$ be $n$ dimensional integral current spaces with weight one and $\bdry T_j=0$. Suppose that $(X_j,d_j)$ are closed Alexandrov spaces with curvature bounded below by $\kappa \in \mathbb R$ and $\diam(X_j) \leq D$.  Then either the sequence converges to the zero integral current space in the intrinsic flat sense 
\be 
(X_j,d_j,T_j) \Fto {\bf 0}
\ee
or a subsequence converges in the Gromov-Hausdorff sense and intrinsic flat sense to the same space:
\be
(X_{j_k},d_{j_k}) \GHto (X,d)
\ee
and
\be
(X_{j_k},d_{j_k},T_{j_k}) \Fto (X,d,T).
\ee
\end{thm}

The proof of Theorem  \ref{thm-mainthm} is an adaptation to Alexandrov spaces of Sormani-Wenger's proof  of the same result for Riemannian manifolds with nonnegative Ricci curvature.  Hence,  to overcome the potential problem coming from the fact that the mass measure is lower semicontinuous with respect to  intrinsic flat distance we work with the notion of filling volume which is continuous with respect to the aforementioned distance. For more details see Appendix  \ref{app}.  The proof applies work of Burago-Gromov-Perelman \cite{BGP},  Portegies-Sormani \cite{PorSor}, Ambrosio-Kirchheim \cite{AK},  Otsu-Shioya \cite{OS} and Sormani-Wenger \cite{CS11}. 

Recently Jaramillo et al \cite{Ja} applying work of Mitsuichi \cite{M2} proved that $n$ dimensional closed oriented Alexandrov spaces  $(X,d)$ can be endowed with a $n$ dimensional integral current structure $T$ with weight one and no boundary such that $(X,d,T)$ is a $n$ dimensional integral current space. This shows the existence of sequences of integral current spaces as in Theorem \ref{thm-mainthm}.  Any compact oriented $n$ dimensional Riemannian manifold $(M,g)$ can easily be endowed with a current structure $T$ coming from the orientation of $M$ such that $(M,d_g, T)$ is a $n$ dimensional integral current space with weight one and $\bdry T = \emptyset$ if $\bdry M= \emptyset$. Honda defined orientability for metric spaces $(X,d)$ arising as (pointed) Gromov-Hausdorff limits of oriented manifolds with a uniform Ricci lower bound. Using the orientation on $(X,d)$, when the limit space is non collapsed, he shows that there is a current $T$ on $X$ such that $(X,d,T)$ is an integral current space with weight one and no boundary \cite{Ho}. Hence, obtaining integral current spaces as in Theorem \ref{thm-mainthm}
in the case the limit space comes from noncollapsing sequences of closed oriented manifolds with sectional curvature uniformly bounded.  

We conjecture that for oriented compact Alexandrov spaces with boundary and nonnegative curvature Theorem \ref{thm-mainthm} also holds replacing $\bdry T_j=0$ by $\bdry (X_j,d_j,T_j)$ is a $n-1$ dimensional integral current space with weight $1$ and $\set (\bdry T_j)= \bdry X_j$.

%


This article is organized as follows, in Section \ref{sec-AlexandrovSpaces} we give the definition of Alexandrov spaces and state the Bishop-Gromov Volume Comparison theorem for them. In Section \ref{sec-GHconvergence} we review Gromov-Hausdorff convergence of sequences of Alexandrov spaces. In Section \ref{sec-Currents} we define integral current spaces. In Section \ref{sec-IFconvergence} we cover intrinsic flat distance. In  Section \ref{sec-proofMainT}, we calculate upper estimates for the mass of the currents and then prove Theorem  \ref{thm-mainthm}. In Appendix \ref{app},  we show that the filling volume of the  standard $n$ dimensional integral current space coming from a $n$ dimensional sphere of radius $r>0$ in Euclidean space equals $r^n$ times the filling volume of the $n$ dimensional integral current space coming from the $n$ dimensional sphere of radius $1$.

We would like to thank Yuri Burago for pointing out
the key theorems needed from Alexandrov Geometry: theorems
reviewed in the tenth chapter of his textbook with Dimitri
Burago and Sergei Ivanov \cite{BBI}. We also would like to thank Christina Sormani for all her suggestions. We thank the ICMS for the Summer School for Ricci curvature: limit spaces and Kahler geometry where we began the project. We also thank the Fields Institute where we finalized this project.  The second author thanks the MSRI for having her as a visitor, the Stony Brook Math Department for the Summer Research Award granted to her, and J. Portegies for his insightful remarks.


\section{Alexandrov Spaces}\label{sec-AlexandrovSpaces}

For equivalent definitions and further information see  Alexander-Kapovitch-Petrunin
 \cite{AKP} and Burago-Burago-Ivanov \cite{BBI}.

Let $\mathbb M_\kappa^n$ denote the $n$ dimensional complete simply connected space of constant sectional curvature equal to $\kappa$. Given three points $a,b$ and $c$ in a length metric space $X$, the triangle 
$\triangle_\kappa \tilde a\tilde b\tilde c  \subset \mathbb M_\kappa^n$ that satisfies
\be
d(a,b)=d(\tilde a, \tilde b),\,  \,\,\,\, d(a,c)=d(\tilde a, \tilde c)\,\,\,\text{and}\,\,d(b,c)=d(\tilde b, \tilde c)
\ee
is called a comparison triangle. We denote by  $\tang{\kappa}{a}{b}{c}$ the angle of $\triangle_\kappa \tilde a\tilde b\tilde c$ at $\tilde a$.

\begin{defn}[From {\cite{BGP}}]
  A complete length metric space $X$ is said to be an Alexandrov space of curvature greater or equal than $\kappa$ if for any $x\in X$ there exists an open neighborhood $U_x$ such that for any four points $p, a, b, c \in U_x$ the quadruple condition holds. 
Namely,  
\begin{equation}
\tang{\kappa}{p}{a}{b}+\tang{\kappa}{p}{a}{c}+\tang{\kappa}{p}{b}{c}\le2\pi.
\end{equation}
We denote by $\Alexnk$ the class of  Alexandrov spaces with curvature bounded from below by $\kappa$ and Hausdorff dimension $\leq n$ and set $\Alex^n(\kappa,D)=\{X\in\Alexnk\,|\, \diam(X)\le D\}$.
\end{defn}

The following Proposition appears as a comment in Section 10.1 in \cite{BBI}.

\begin{prop}\label{prop-F&G}
Let $D_\kappa < \diam (\mathbb M_\kappa^n)/2$. The function $F_\kappa: (0, D_\kappa)^3 \to \R$ given  by 
\be 
F_\kappa(d_1,d_2,d_3) = \tang{\kappa}{\tilde a}{ \tilde b}{\tilde c},
\ee
where  $\tilde a, \tilde b, \tilde c  \in \mathbb M_\kappa^n$ such that 
\be
d(\tilde a, \tilde b)=d_1,\,  \,\,\,\, d(\tilde a, \tilde c)=d_2\,\,\,\text{and}\,\,d(\tilde b, \tilde c)=d_3
\ee
is continuous. 
\end{prop}

\begin{proof}
There exists a geodesic triangle in $\mathbb M_\kappa^n$ with length sides $d_1,d_2,d_3$ provided 
$d_1+d_2+d_3 < 2D_\kappa$.  Hence, $F_\kappa$ is well defined. 
To see that it is continuous,  it is enough to recall the law of cosines for $\kappa=1,0,-1$,  
\be
\cos(d_3)= \cos(d_1)\cos(d_2) +  \sin(d_1)\sin(d_2)\cos( \tang{1}{\tilde a}{ \tilde b}{\tilde c} ),
\ee
\be
d_3^2= d_1^2 + d_2^2 - 2d_1d_2\cos( \tang{0}{\tilde a}{ \tilde b}{\tilde c})
\ee
and 
\be
\cosh(d_3)= \cosh(d_1)\cosh(d_2) -  \sinh(d_1)\sinh(d_2)\cosh( \tang{-1}{\tilde a}{ \tilde b}{\tilde c} ).
\ee
These functions are continuous as functions of $(d_1,d_2,d_3)$ and  the functions that multiply $\cos( \tang{k}{\tilde a}{\tilde b} {\tilde c} )$ are never zero for $d_i \in  (0, D_\kappa)$.  This concludes the proof. 
\end{proof}

Burago-Gromov-Perelman proved that Bishop-Gromov Volume Comparison holds for Alexandrov spaces.
 
\begin{thm}[Burago-Gromov-Perelman, 10.13 in \cite{BGP}, c.f. Theorem 10.6.6 in \cite{BBI}]\label{thm-BishopGromov}
Let $X$ be a $n$ dimensional  Alexandrov space of curvature $\geq \kappa$.  Let $V_{n,\kappa} (r)$ denote the volume of a ball of radius $r$ in $\mathbb M_\kappa^n$. Then for all $x \in X$ the function, 
\be
r \,\,\, \,\, \longmapsto \,\, \frac {\mathcal{H}^n(B_r(x))}  {V_{n,\kappa} (r)},
\ee
is non-increasing. 
\end{thm}


\section{Gromov-Hausdorff Convergence of Alexandrov Spaces}\label{sec-GHconvergence}

Given a complete metric space $Z$, the Hausdorff distance between two subsets $A$ and $B$ of $Z$ is given by
\be
d_{H}^Z\left(A,B\right) = \inf\{ \varepsilon>0\,|\, A \subset T_\varepsilon\left(B\right) \textrm{ and } B \subset T_\varepsilon\left(A\right)\},
\ee
where $T_\varepsilon\left(A\right)$ denotes the $\varepsilon$-tubular neighborhood of $A$. 

The Hausdorff distance is generalized to metric spaces $(X_i, d_i)$ in the following way.

\begin{defn}[Gromov]\label{defn-GH}
Let $\left(X_i, d_{X_i}\right)$, $i=1,2$, be two metric spaces. The Gromov-Hausdorff distance between them is defined as
\be \label{eqn-GH-def}
d_{GH}\left( X_1,X_2 \right) = \inf \{d^Z_H\left(\varphi_1 \left(X_1\right), \varphi_2\left(X_2\right)\right)\}
\ee
where the infimum is taken over all complete metric spaces $Z$ and all isometric embeddings $\varphi_i: X_i \to Z$.
\end{defn}

The function $d_{GH}$ is symmetric and satisfies the triangle inequality.
It becomes a distance when restricted to compact metric spaces.  The following compactness theorem is due to Gromov.

\begin{thm}[Gromov]\label{thm-GromovCompactness}
Let $(X_j, d_j)$ be a sequence of compact metric spaces. If there exist $D >0$, and $N:(0,\infty) \to \N$  such that for all $j$
\be
\diam(X_j) \leq D
\ee
and for all $\varepsilon$ there are $N(\varepsilon)$  $\varepsilon$-balls that cover $X_j$, then 
a subsequence of $X_j$  converges to a compact metric space $(X_\infty, d_\infty)$ in Gromov-Hausdorff sense,
\be
(X_{j_k}, d_{j_k}) \GHto (X_\infty, d_\infty).
\ee
\end{thm}

In  \cite{Gromov} Gromov proved that if a sequence  of compact metric spaces $X_j$ converges in Gromov-Hausdorff sense to $X_\infty$ then there is a compact metric space $Z$ and isometric embeddings 
$\varphi_j : X_j \to Z$ such that for a subsequence 
\be 
d^Z_H(\varphi_j (X_j), \varphi_\infty (X_\infty)) \to 0.
\ee
Thus we say that a sequence of points $x_j \in X_j$ converges to $x \in X_\infty$ if
$\varphi_j(x_j)$ converges to $\varphi_\infty(x)$ in $Z$. 

Recall that  $\AlexnkD$ denotes the class of Alexandrov spaces with curvature $\geq \kappa$, diameter $\leq D$ and Hausdorff dimension $\leq n$. Gromov's compactness theorem can be applied to prove convergence of sequences of Alexandrov spaces.

\begin{thm}[Burago-Gromov-Perelman, 8.5 in \cite{BGP}]\label{thm-alexComp}
Let $X_j$ be a sequence of compact Alexandrov spaces contained in $\AlexnkD$. Then there is a subsequence $X_{j_k}$ and a compact Alexandrov space $X \in \AlexnkD$ such that 
\be 
X_{j_k} \GHto X.
\ee
\end{thm}

\begin{thm}[Burago-Gromov-Perelman, 10.8 in \cite{BGP}, cf. Corollary 10.10.11 in \cite{BBI}]\label{cor-collapsingSeq}
Let $X_j$ be a sequence of Alexandrov spaces contained in $\AlexnkD$ that converges in Gromov-Hausdorff sense to  $X$.  Then $\mathcal \lim_{j \to \infty} \mathcal{H}^n(X_j)=0$ if and only if $\dim_H X < n$.
\end{thm}

Let $X$ be a $n$ dimensional Alexandrov space and $x \in X$.  If for any $\varepsilon>0$ there is $r>0$ such that 
\be 
d_{GH}(B_r(x),B_r(0))<\varepsilon r,
\ee
where $B_r(x)$ denotes the ball in X of radius $r$ centered at $x$  and $B_r(0)$ the ball in $\Bbb R^n$ of radius $r$ centered at $0$, then $x$ is called a regular point of $X$.  We denote by $R(X)$ the set of all regular points of $X$.

\begin{thm}[Corollary 10.9.13 in \cite{BBI}]\label{thm-denseRegSet}
Let $X$ be a $n$ dimensional Alexandrov space. Then $R(X)$ is dense in $X$.
\end{thm}

\begin{thm}[Otsu-Shioya, Theorem A in \cite{OS}]\label{thm-dimSingSet}
Let $X$ be a $n$ dimensional Alexandrov space. Then $X\setminus R(X)$ has Hausdorff dimension less or equal than $n-1$.
\end{thm}

Equivalently,   $x$ is a regular point of $X$  if for each $\delta > 0$ there is a $(n,\delta)$ strainer for $x$.  A $(n,\delta)$ strainer is a collection of points $\{(a_i,b_i)\}_{i=1}^n \subset X \times X$ that for $i\neq j$
\be
\tang{\kappa}{x}{a_i}{a_j}>\pi/2-\delta,\,\,\,\,\,\,\tang{\kappa}{x}{b_i}{b_j}>\pi/2-\delta,\,\,\tang{\kappa}{x}{a_i}{b_j}>\pi/2-\delta
\ee
and, for all $i$
\be
\tang{\kappa}{x}{a_i}{b_i}>\pi-\delta.
\ee

\begin{lem}\label{lem-d/2strainer}
Let $X$ be an Alexandrov space and $\{(a_i,b_i)\}_{i=1}^n$  a $(n,\delta)$ strainer for $x \in X$. Then,  there is $\varepsilon >0$
such that if $x', a'_i,b'_i \in X$, $i=1,...,n$,  satisfy
  \begin{align}
|d(a'_i,a'_k) - d(a_i,a_k)  |  <    \varepsilon \\ 
|d(b'_i,b'_k)  - (b_i,b_k)|     <   \varepsilon  \\
|d(a'_i,b'_i)  -  d(a_i,b_i) |   <  \varepsilon \\ 
|d(x', a'_i)  -  d(x, a_i)|   <    \varepsilon    \\
|d(x', b'_i)  -  d(x, b_i)| <  \varepsilon,
\end{align}
then $\{(a'_i,b'_i)\}_{i=1}^n$ is a $(n,\delta/2)$ strainer for $x'$.

\end{lem}

\begin{proof}
Apply Proposition \ref{prop-F&G} several times and that $\{(a_i,b_i)\}_{i=1}^n$  is a $(n,\delta)$ strainer for $x \in X$ to get 
$\varepsilon > 0$ such that for points $x', a'_i,b'_i \in X$, $i=1,...,n$ satisfying the hypotheses, then for $i \neq k$
\begin{align}
F_\kappa(d(x', a'_i), d(x',a'_k),d(a'_i,a'_k))= & \tang{\kappa}{x'}{a'_i}{a'_k}   > \pi/2- \delta/2\\
F_\kappa(d(x', b'_i), d(x',b'_k),d(b'_i,b'_k))=  & \tang{\kappa}{x'}{b'_i}{b'_k}  > \pi/2- \delta/2 \\ 
F_\kappa(d(x', a'_i), d(x',b'_k),d(a'_i,b'_k))=  & \tang{\kappa}{x'}{a'_i}{b'_k}  > \pi/2- \delta/2 \\
F_\kappa(d(x', a'_i), d(y,b'_i),d(a'_i,b'_i))=  & \tang{\kappa}{x'}{a'_i}{b'_i} > \pi- \delta/2.
\end{align}
\end{proof}

Given a $(n, \delta)$ strainer $\{(a_i,b_i)\}_{i=1}^n$ define $f : X \to \R^n$ by 
\begin{equation}
f (x)= (d( x,a_1),...,d(x,a_n)).
\end{equation}

\begin{thm}[Proposition 10.8.15 and Theorem 10.8.18 in \cite{BBI}]\label{thm-fbi}
Let $X$ be a $n$ dimensional Alexandrov space and  $\{(a_i,b_i)\}_{i=1}^n$ a $(n, \delta)$ strainer
for $x \in X$ with $\delta \leq 1/100n$.  Then there exists $U \subset X$ open, $x \in U$, such that for all $y \in U$ 
\begin{enumerate}
\item $\{(a_i,b_i)\}_{i=1}^n$ is also a $(n, \delta)$ strainer for $y$
\item for some $l >0$ and all $i$, $d(y,a_i), d(y,b_i) > l$.
\end{enumerate}
Furthermore, for any neighborhood $U$ of $x$ that satisfies (1) and $(2)$,  
$f|_{U}: U \to f(U)$ is a bilipschitz function with $\Lip(f) \leq \sqrt{n}$ and $\Lip(f^{-1})\leq 500n$.
\end{thm}

We combine the previous results to obtain a sequence of bilipschitz functions. 

\begin{lem}\label{lem-lipmaps}
Let $(X_j,d_j) \in \AlexnkD$ be a noncollapsing sequence of compact Alexandrov spaces that converges in Gromov-Hausdorff sense to $X$. Let $x \in R(X)$, then there exist a sequence of points $x_j \in X_j \to x$, $r_0 > 0$, open sets $W_j \subset \R^n$ and bilipschitz functions 
\begin{equation}
f_j: B_{r_0}(x_j) \to W_j
\end{equation}
such that $\Lip(f_j)\leq L$,   $\Lip(f^{-1}_j)\leq L$. 
\end{lem}

\begin{proof}
Take $\delta \leq 1/200n$. Since  $x \in R(X)$ there exist  a $(n, 1/200n)$ strainer  $\{(a_i,b_i)\}_{i=1}^n$ for $x$ and a function
$f : U \to W$ given by
\begin{equation}
f(x)= (d(x,a_1),...,d(x,a_n)),
\end{equation}
where $U$ satisfies (1) and (2) of Theorem \ref{thm-fbi}.

By Gromov's embedding theorem we can assume that all $X_j$ are contained in a compact metric space $(Z,d)$ in which they Hausdorff converge to X,
\begin{equation}\label{eq-Hconv}
d_H^Z(X_j,X) < \varepsilon_j  \,\,\,\text{such that}\,\,\, \varepsilon_j  \downarrow 0.
\end{equation}
Thus, there exist $x_j \in X_j$ and $\{(a^j_i,b^j_i)\}_{i=1}^n \subset X_j \times X_j$ such that 
$d(x_j, x) < \varepsilon_j$, $d(a^j_i,a_i) < \varepsilon_j$ and $d(b^j_i,b_i) < \varepsilon_j$.  Define $f_j: X_j \to \R^n$ by 
\begin{equation}
f_j(x)= (d( x ,a^j_1),...,d(x,a^j_n)).
\end{equation}

We claim that there is $r_0 > 0$ such that for $j$ big enough each $\{(a^j_i,b^j_i)\}_{i=1}^n$ satisfies conditions (1) and (2) of Theorem \ref{thm-fbi} with $U=B_{r_0}(x_j)$. First we show that there is $r >0$ such that for $j$ big enough, $\{(a^j_i,b^j_i)\}_{i=1}^n$ is a $(n, 1/100n)$ strainer for all $y \in B_r(x_j)$.
To obtain the estimates needed we apply Lemma \ref{lem-d/2strainer}. By the triangle inequality,
\begin{align}
|d(a^j_i,a^j_k) - d(a_i,a_k)  | \leq    d(a^j_i,a_i) + d(a_k,a^j_k)      <    2\varepsilon_j  \\ 
|d(b^j_i,b^j_k)  - (b_i,b_k)|     \leq  d(b^j_i,b_i)  + d(b_k,b^j_k)       <   2\varepsilon_j   \\
|d(a^j_i,b^j_i)  -  d(a_i,b_i) |  \leq    d(a^j_i,a_i)  -  d(b_i,b^j_i)     <  2\varepsilon_j .  
\end{align}

Since $U$ is open  and $x \in U$,  take $\overline{B_{2r}(x)} \subset U$.  By equation (\ref{eq-Hconv}), for $y \in B_r(x_j)$ there exist 
$y' \in X$ such that $d(y,y') < \varepsilon_j$.  By the triangle inequality,  $d(y',x) \leq d(y',y) + d(y,x_j) + d(x_j,x) < r + 2 \varepsilon_j$. Thus $y' \in B_{r + 2\varepsilon_j}(x)$ and 
\begin{align}
|d(y, a^j_i)  -  d(y', a_i)|  \leq  d(y,y')   + d(a^j_i, a_i) <    2\varepsilon_j    \\
|d(y, b^j_i)  -  d(y', b_i)| \leq   d(y, y')  +  d(b^j_i, b_i) <  2\varepsilon_j.  
\end{align}
Since $\varepsilon_j  \downarrow 0$ we can take $j$ big enough such that $B_{r + 2\varepsilon_j}(x) \subset B_{2r}(x)$, thus $y' \in B_{2r}(x) \subset U$.  Since $y' \in U$, from condition (1) of Theorem \ref{thm-fbi}, $\{(a_i,b_i)\}_{i=1}^n$ is a $(n, 1/200n)$ strainer for  $y'$.  Then by Lemma \ref{lem-d/2strainer}, for $j$ big enough $\{(a^j_i,b^j_i)\}_{i=1}^n$ is a $(n, 1/100n)$ strainer for $y$. This shows that for $j$ big enough, condition (1) of Theorem \ref{thm-fbi} is satisfied for any $y \in B_r(x_j)$.

To prove that condition (2) holds we choose $r_0 \leq r$ in the following way. 
Let $m= \min_{1 \leq i \leq n} \{ d(x,a_i),d(x,b_i)  \}$. Recall that (2) of Theorem \ref{thm-fbi} holds for $\{(a_i,b_i)\}_{i=1}^n$. Then, $m > l >0$. Define 
\begin{equation}
r_0 = \min\{   \tfrac1 2 (m - l), r\} >0.
\end{equation}
Take $y \in B_{r_0}(x_j)$. Using the triangle inequality, the definition of $r_0$ and,  $d(x,a_i)>m$, $d(x_j, x) < \varepsilon_j$ and  $d(a^j_i,a_i) < \varepsilon_j$ for all $i,j$, we obtain
\begin{align}
d(y,a^j_i) & \geq d(x,a_i) - d(y,x_j) - d(x_j, x) - d(a_i, a^j_i)  \\
&    \geq  m - \tfrac1 2 (m -  l)  - 2\varepsilon_j= l + \tfrac1 2 (m -  l)  - 2\varepsilon_j .
\end{align}
Thus, for $j$ big enough $d(y,a^j_i) > l$ for all $y \in B_{r_0}(x_j)$. Hence, for $j$ big enough condition (2) of Theorem \ref{thm-fbi} is satisfied with $\{(a^j_i,b^j_i)\}_{i=1}^n$ as $(n,\delta)$ strainer of $x_j$ and $B_{r_0}(x_j)$ as its neighborhood.  The proof finishes applying Theorem \ref{thm-fbi}. 
\end{proof}


\section{Integral Current Spaces}\label{sec-Currents}

In this section we review integral current spaces as presented in Section 2 of Sormani-Wenger \cite{CS11}.  We refer to their work for a complete exposition of integral current spaces. For readers interested in integral currents we refer to  Ambrosio-Kirchheim's paper \cite{AK} and \cite{CS11} for integral current spaces. 

Let $Z$ be a metric space. We denote by $\mathcal{D}^m(Z)$ the collection
of $(m+1)$-tuples of Lipschitz functions such that the first entry function is bounded:
\be
\mathcal{D}^m(Z)=\left\{ (f,\pi)=\left(f,\pi_1 ..., \pi_m\right)\, |\, f, \pi_i: Z \to \R \text{ Lipschitz and } f \,\text{ is bounded}\right\}.
\ee

\begin{defn}[Ambrosio-Kirchheim]
Let $Z$ be a complete metric space. A multilinear functional $T:\mathcal{D}^m(Z) \to \R$ is called an $m$ dimensional current if it satisfies:

i) If there is an $i$ such that $\pi_i$ is constant on a neighborhood of $\{f\neq0\}$ then $T(f, \pi)=0$.

ii) $T$ is continuous with respect to the pointwise convergence of the $\pi_i$ for  $\Lip(\pi_i)\le 1$.

iii) There exists a finite Borel measure $\mu$ on $Z$ such that for all $(f,\pi)\in \mathcal{D}^m(Z)$
\be\label{def-AK-current-iii}
|T(f,\pi)| \le \prod_{i=1}^m \Lip(\pi_i)  \int_Z |f| \,d\mu .
\ee
The collection of all m dimensional currents of $Z$ is denoted by $\curr_m(Z)$.
\end{defn}

\begin{defn}[Ambrosio-Kirchheim] \label{defn-mass}
Let $T:\mathcal{D}^m(Z) \to \R$ be an $m$-dimensional current. The mass measure of $T$ is the smallest Borel measure $\|T\|$  such that (\ref{def-AK-current-iii}) holds for all $(f,\pi) \in \mathcal{D}^m(Z)$.

The mass of $T$ is defined as
\be \label{def-mass-from-current}
\mass \left(T\right) = || T || \left(Z\right) = \int_Z \, d\| T\|.
\ee
\end{defn}

The most studied currents come from pushforwards and Example \ref{basic-current-pushed}.

\begin{defn}[Ambrosio-Kirchheim Defn 2.4]
Let $T\in \curr_m(Z)$ and $\varphi:Z\to Z'$ be a
Lipschitz map. The {\em pushforward} of $T$
to a current $\varphi_\# T \in \curr_m(Z')$ is given by
\be \label{def-push-forward}
\varphi_\#T(f,\pi)=T(f\circ \varphi, \pi_1\circ\varphi,..., \pi_m\circ\varphi).
\ee
\end{defn}

\begin{ex}[Ambrosio-Kirchheim]\label{basic-current}
Let $h: A \subset \R^m \to \Z$ be an $L^1$ function. Then $\Lbrack h \Rbrack:  \mathcal{D}^m(\R^m) \to \R$ given by
\be \label{def-current-from-function}
\Lbrack h \Rbrack \left(f, \pi\right) = \int_{A \subset \R^m}  h f \det\left(\nabla \pi_i\right) \, d\mathcal{L}^m
\ee
is an $m$ dimensional current with $|| T || (A) = \int_A |h| d\mathcal{L}^m$.
\end{ex}

\begin{ex}[Ambrosio-Kirchheim]\label{basic-current-pushed}
Let $h: A \subset \R^m \to \Z$ be an $L^1$ function and
$\varphi:  \R^m \to Z$ be a
bilipschitz map, then $\varphi_\# \Lbrack h \Rbrack \in \curr_m(Z)$. Explicitly,
\be
\varphi_\# \Lbrack h \Rbrack (f,\pi_1,..., \pi_m)=
\int_{A \subset \R^m} h \, (f\circ\varphi) \det\left(\nabla (\pi_i \circ \varphi)\right) \, d\mathcal{L}^m .
\ee
\end{ex}

\begin{rmrk}
$\nabla \pi_i$, $\nabla (\pi_i \circ \varphi)$ are defined almost everywhere by Rademacher's Theorem.
\end{rmrk}

Now we are ready to define integer rectifiable currents.

\begin{defn}[Definition 4.2, Theorem 4.5 in Ambrosio-Kirchheim \cite{AK}] \label{def-param-rep} 
Let $T\in \curr_m(Z)$. $T$ is an  integer rectifiable
current if it has a parametrization of the form $\left(\{\varphi_i\}, \{\theta_i\}\right)$, where

i) $\varphi_i:A_i\subset\R^m \to Z$ is a countable collection of bilipschitz maps
with $A_i$ precompact Borel measurable sets with pairwise disjoint images

ii) $\theta_i\in L^1\left(A_i,\N\right)$ and the following equations are satisfied 

\be\label{param-representation}
T = \sum_{i=1}^\infty \varphi_{i\#} \Lbrack \theta_i \Rbrack \quad\text{and}\quad \mass\left(T\right) = \sum_{i=1}^\infty \mass\left(\varphi_{i\#}\Lbrack \theta_i \Rbrack\right).
\ee
The space of $m$ dimensional integer rectifiable currents on $Z$ is denoted by
$\intrectcurr_m\left(Z\right)$.
\end{defn}

We note that in the previous case the mass measure of $T$ can be rewritten as
\be
||T|| = \sum_{i=1}^\infty ||\varphi_{i\#}\Lbrack \theta_i \Rbrack ||.
\ee

We define the boundary operator and the canonical set of a current. 

\begin{defn}[Ambrosio-Kirchheim]
An $m$ dimensional integral current is an integer rectifiable current,  $ T\in\intrectcurr_m(Z)$, such that $\partial T$
defined as
\begin{equation}
\partial T \left(f, \pi_1,..., \pi_{m-1}\right) = T \left(1, f, \pi_1,..., \pi_{m-1}\right)
\end{equation}
is a current. We denote by $\intcurr_m\left(Z\right)$ the space of $m$ dimensional integral currents on $Z$.
\end{defn}

\begin{defn}[Ambrosio-Kirchheim] \label{defn-set}
Let $T \in \curr_m(Z)$, the canonical set of $T$ is the set
\be \label{def-set-current}
\set\left(T\right)= \{p \in Z: \Theta_{*m}\left( \|T\|, p\right) >0\},
\ee
where $\Theta_{*m}\left( \|T\|, p\right)$ is the lower $m$ dimensional density of $\|T\|$ at $p \in Z$.
\be 
\Theta_{*m}\left( \|T\|, p\right)= \liminf_{r\to 0} \frac{\|T\|(B_r(p))}{\omega_m r^m}.
\ee
Here $\omega_m$ denotes the volume of the unit ball in $\R^m$.
\end{defn}

\begin{defn} [Sormani-Wenger]\label{defn-integral-current-space}
We say that $(X,d,T)$ is an $m$ dimensional integral current space if either $\left(X,d\right)$ is a metric space and $T \in \intcurr_m(\overline{X})$ with $\set\left(T\right)=X$ or $(X,d,T)=(\emptyset,0,0)$.

We say that $(\emptyset,0,0)$ is the zero integral current space and denote it by $\bf 0$. Moreover, we denote by $\mathcal{M}^m$  the space of $m$ dimensional integral current spaces and by $\mathcal{M}_0^m$ the space of $m$ dimensional integral current spaces whose canonical set is precompact. 
\end{defn}

Ambrosio-Kirchheim proved in \cite{AK} that if $T \in \intrectcurr_m\left(Z\right)$ then $\set(T)$ is a countably $\mathcal{H}^m$ rectifiable metric space.  Hence,  for any $m$ dimensional integral current space $(X,d,T)$,  $X$ is countably $\mathcal{H}^m$ rectifiable. Ambrosio-Kirchheim also characterized the mass measure.

\begin{defn}\label{def-weight}
Let $T \in \intrectcurr_m\left(Z\right)$ with parametrization $\left(\{\varphi_i\}, \theta_i\right)$. Then the function $\theta_T: Z \to \N \cup \{0\}$ given by
\be\label{weight}
\theta_T= \sum_{i=1}^\infty \theta_i\circ\varphi_i^{-1}\One_{\varphi_i\left(A_i\right)}
\ee
is called the weight of $T$.
\end{defn}

\begin{lem}[Ambrosio-Kirchheim]\label{lemma-weight}
Let $T \in \intrectcurr_m\left(Z\right)$. Then there is a function
\be \label{eqn-lem-weight-lambda}
\lambda:\set(T) \to [m^{-m/2}, 2^m/\omega_m]
\ee
such that
\be \label{eqn-lem-weight-new-key}
\Theta_{*m}(||T||,x)=\theta_T(x)\lambda(x) 
\ee
for $\mathcal{H}^m$ almost every $x\in \set (T)$ and
\be \label{eqn-lem-weight-2}
||T||=\theta_T \lambda \mathcal{H}^m \rstr \set(T),
\ee
where $\omega_m$ denotes the volume of an unitary ball in $\R^m$ and $\theta_T$ is the weight of $T$ as in Definition \ref{def-weight}. 
\end{lem}


\section{Intrinsic Flat Convergence}\label{sec-IFconvergence}

In this section we give the definition of intrinsic flat distance between integral current spaces and state various intrinsic flat convergence theorems. 
Most of the results stated here can be found in \cite{CS11}.

Let  $S_i \in \intcurr_{m}(Z)$, $i=1,2$, then the flat distance between them 
is given by 
\begin{equation}\label{equation:def-abstract-flat-distance}
d_{F}\left(S_1,S_2\right)= \inf\{\mass\left(U\right)+\mass\left(V\right)\},
\end{equation}
where the infimum is taken over all
 integral currents
$U\in\intcurr_m\left(Z\right)$ and $V\in\intcurr_{m+1}\left(Z\right)$
such that
\begin{equation} \label{eqn-Federer-Flat-2}
S_1- S_2 =U+\bdry V.
\end{equation}

\begin{defn}[Sormani-Wenger]\label{defn-IFdistance}
Let $(X_i,d_i,T_i) \in \mathcal{M}^m$. Then,
\begin{equation}\label{eqn-local-defn}
d_{\Fm}\left((X_1,d_1,T_1),(X_2,d_2,T_2)\right)= \inf d_F^Z
\left(\varphi_{1\#} T_1, \varphi_{2\#} T_2 \right), 
\end{equation}
where the infimum is taken over all complete metric spaces,
$\left(Z,d\right)$, and all isometric embeddings
\begin{equation}
\varphi_i : \left(\bar{X_i}, d_i\right)\to \left(Z,d\right).
\end{equation}
Note that in this definition we also include the zero integral current space ${\bf{0}}$. Its current, $0$, isometrically embeds into any $Z$ as $\varphi_\#0=0 \in \intcurr_m\left(Z\right)$.
\end{defn}

If one of the integral currents spaces is the zero integral current space, say $(X_2,d_2,T_2)={\bf{0}}$, we can choose $V=0$  and $U=T_1$. Then  by definition
\be
d_{\Fm}\left((X_1,d_1,T_1),{\bf{0}})\right) \leq \mass\left(T_1\right).
\ee
This estimate will be used  in the proof of Theorem \ref{thm-mainthm}  to calculate the lower $n$ dimensional  density of $\|T\|$ at points in the Gromov-Hausdorff limit space. 

In Theorem 3.27 of \cite{CS11} is proven that $d_{\Fm}$ is a distance on the class of precompact integral currents spaces, $\mathcal{M}_0^m$.

\begin{thm}[Sormani-Wenger]\label{thm-IFinsideGH}
Let $(X_j, d_j, T_j)$ be a sequence of $m$ dimensional integral current spaces. If there exist $D$, $M$ and $N:(0,\infty) \to \N$  such that for all $j$
\be
\diam(X_j) \leq D, \,\,\, \mass(T_j) + \mass(\bdry T_j) \leq M
\ee
and, for all $\varepsilon$ there are $N(\varepsilon)$  $\varepsilon$-balls that cover $X_j$, then
\be
(X_{j_k}, d_{j_k}) \GHto (X,d)\,\,\,\text{and}\,\,
(X_{j_k}, d_{j_k}, T_{j_k}) \Fto (Y,d_Y,T),
\ee
where $\left(Y,d_Y,T\right)$ is an $m$ dimensional integral current space either equal to  ${\bf 0}$ or
$Y \subset X$.
\end{thm}

\begin{rmrk}\label{rmrk-IFinsideGH}
When we apply Theorem \ref{thm-IFinsideGH} we generally consider that all the spaces are embedded in a complete metric space. This assumption comes from its proof. Sormani-Wenger applied Gromov's embedding theorem \cite{Gromov} to obtain 
a complete metric space $Z$ and isometric embeddings $\varphi_j : X_j \to Z$, 
$\varphi : X \to Z$  such that for a subsequence 
\be 
d^Z_H(\varphi_j (X_j), \varphi (X)) \to 0.
\ee
Then they show that a further subsequence satisfies
\be 
d^Z_F (\varphi_{j\#} T_j, \varphi_{\#} T) \to 0.
\ee
\end{rmrk}


Given $T \in \curr_m(Z)$
and a Borel set $A$, the restriction of $T$ to $A$ is a current  $T \rstr A \in  \curr_m(Z)$ given by 
\be
\left( T \rstr A \right) (f, \pi)=T( 1_A f , \pi),
\ee
where $1_A$ is the indicator function of $A$. See Definition 2.19 in \cite{CS11}. The mass of $T \rstr A$
is $||T||(A)$.


\begin{thm}[Sormani \cite{SorArzela} ]\label{thm-SorArzela}
Let $(X_j, d_j, T_j)$ be a sequence of  integral current spaces such that 
\be
(X_j, d_j, T_j) \Fto (X_\infty ,d_\infty,T_\infty ).
\ee
Then for almost  all $r>0$,
\be
(\bar B_{ r}(x_j) , d_j, T_j \rstr \bar B_{r}(x_j) )
\ee
are integral current spaces. 
\end{thm}


\begin{lem}[Sormani-Wenger \cite{CS11}]\label{lem-distFlatLip}
Let $(X_i,d_i)$ be complete metric spaces and $\varphi: X_1 \to X_2$ be a $\lambda >1$ bilipschitz map. If $T \in \intcurr_m(X_1)$, then for $M_1=(\set(T), d_1,T)$ and  $M_2=(\set(\varphi_\#T), d_2,\varphi_\# T)$ the following inequality holds
\be
d_{\Fm}\left(M_1,M_2\right) \leq k_{\lambda,m} \max\{\diam(\spt T), \diam(\varphi(\spt  T))\} \mass(T)
\ee
where $k_{\lambda,m}=\frac12(m+1)\lambda^{m-1}(\lambda-1)$.
\end{lem}


\section{Proof of Theorem 0.1}\label{sec-proofMainT}

We describe now the proof of Theorem  \ref{thm-mainthm}.  First, from Sormani-Wenger's theorem, Theorem   \ref{thm-IFinsideGH}, we obtain a sequence that converges in Gromov-Hausdorff sense to $(X,d)$ and intrinsic flat sense to $(Y, d_Y, T)$.  If the intrinsic flat limit is non zero then $X$ has non zero $n$ Hausdorff measure by Ambrosio-Kirchheim's characterization of the mass measure,  Lemma \ref{lemma-weight}.  Then we prove that the set of regular points of $X$, $R(X)$, is contained in $Y$.  If $x$ is a regular point of $X$ we consider a sequence of points $x_j \in X_j \to x$. Applying Lemma \ref{lem-lipmaps} we obtain bilipchitz maps from $B_{r_0}(x_j)$ to $W_j \subset \R^n$ with uniform Lipschitz constants.   The existence of these maps provide an estimate on the intrinsic flat distance between the integral current spaces defined on $B_{r/L}(x_j)$ and $f_j(B_{r/L}(x_j))$ (cf. Sormani-Wenger's  lemma, Lemma \ref{lem-distFlatLip}).  From these estimates we obtain bounds of the form  $||T||(B_r(x)) \geq C(n,r)r^n$ such that $\lim_{r\to 0} C(n,r)> 0$.  This shows $x \in Y$ and thus  $R(X)$ is contained in $Y$.  To see that the singular points of $X$, $X \setminus R(X)$, are contained in  $Y$ we apply Otsu-Shioya's result about the Hausdorff dimension of the singular set of $X$, Ambrosio-Kirchheim's characterization of the mass measure and, Burago-Gromov-Perelman's Bishop and Bishop-Gromov Volume Comparison theorems for Alexandrov spaces \cite{BGP}.

We first prove the following lemmas. 

\begin{lem}\label{TLeH}
Let $(X,d)$ be a $n$ dimensional Alexandrov space equipped with an integral current structure $T$ with weight equal to one. Then for all $r > 0$
\be 
||T||(B_r(x))\le c(n) \mathcal H^n(B_r(x)).
\ee 
In particular, if $(X,d)\in\AlexnkD$ then
\be
  ||T||(X)\le C(n,\kappa,D)
\ee
 \end{lem}

\begin{proof}
By Lemma \ref{lemma-weight},
\be 
||T||(B_r(x))=\int_{B_r(x)}\theta_T(x)\lambda(x)d\mathcal H^n \leq 2^n/ \omega_n \mathcal H^n(B_r(x)),
\ee
where we used that $\theta_T=1$ and $\lambda \leq 2^n/ \omega_n$.

To finish the proof,  by the Bishop Volume Comparison theorem for Alexandrov spaces (Theorem 10.2 in \cite{BGP})  we know that
\be 
\mathcal H^n(B_r(x))\le\mathcal H^n(B_r(\mathbb M_\kappa^n)),
\ee
where $\mathbb M_\kappa^n$ denotes the $n$-dimensional complete simply connected space of constant sectional curvature $\kappa$ and  $B_r(\mathbb M_\kappa^n)$ a ball in $\mathbb M_\kappa^n$ with radius $r$.
\end{proof}

\begin{lem}\label{lem-standcurr}
Let $(X,d,T)$ be a $n$ dimensional integral current space  with weight equal to 1, $\theta_T=1$ and $f: X \to f(X) \subset \mathbb \R^n$ a bilipschitz map. 
Then $(f(X), d_{\mathbb R^n}, f_{\#}  T)$ has the standard current structure, 
\be 
 f_{\#}  T (h,\pi) =  \Lbrack 1_{f(X)}  \Rbrack \left(h, \pi\right) = \int_{f(X)} h \det\left(\nabla \pi_i\right) \, d\mathcal{L}^n.
\ee
In particular, $||  f_{\#} T || = \mathcal{L}^n \rstr f(X) $.
\end{lem}

\begin{proof}
Since $T \in \curr_n(\bar X)$ there is a parametrization  $\left(\{\varphi_i\}, \{\theta_i\}\right)$ as in Definition \ref{def-param-rep}, such that 
\be
T= \sum_{i=1}^\infty \varphi_{i\#} \Lbrack 1_{A_i} \Rbrack.
\ee
Let $(h,\pi) \in \mathcal D^n(f(X))$. By the previous equality, the definition of the pushforward of a current,  change of variables in $\mathbb R^n$,  and the fact that $f(\varphi_i(A_i))$ are mutually disjoint,
\begin{align}
f_\# T (h,\pi) = & \sum_{i=1}^\infty f_\#( \varphi_{i\#} \Lbrack 1_{A_i} \Rbrack ) (h,\pi)=  \sum_{i=1}^\infty  \int_{A_i \subset \R^n}  h \circ f \circ \varphi_i \det\left(\nabla (\pi_i \circ f \circ \varphi_i )\right) \, 
d\mathcal{L}^n \\
& = \sum_{i=1}^\infty   \int_{ f ( \varphi_i(A_i)) \subset \R^n} h \det\left(\nabla \pi_i \right) \, d\mathcal{L}^n =  \int_{ \union f ( \varphi_i(A_i)) \subset \R^n} h \det\left(\nabla \pi_i \right) \, d\mathcal{L}^n. 
\end{align}

Since $\mathcal H^n(X \setminus \cup \varphi_i(A_i)) =0$ and $f$ is Lipschitz, 
\be
\mathcal H^n(f(X) \setminus \cup   f( \varphi_i(A_i))) \leq \Lip^n(f)  \mathcal H^n(X \setminus \cup \varphi_i(A_i)) =0. 
\ee
 Hence, 
 \be
f_\# T (h,\pi) =  \int_{ f (X)} h \det\left(\nabla \pi_i \right) \, d\mathcal{L}^n.
\ee 
The claim about $||f_\#T ||$ follows from Example \ref{basic-current}.
\end{proof}

\begin{proof}[\bf Proof of Theorem \ref{thm-mainthm}]
By the compactness theorem for Alexandrov spaces, Theorem \ref{thm-alexComp},  there is an Alexandrov space $(X,d)$ with nonnegative curvature, $\diam(X) \leq D$ and Hausdorff dimension $\leq n$ and, a convergent subsequence:
\be
(X_{j_k}, d_{j_k}) \GHto (X,d).
\ee
This, together with the uniform upper bound of  $||T_j||(X_j)$  given by Lemma \ref{TLeH} allow us to apply Sormani-Wenger's theorem, Theorem \ref{thm-IFinsideGH}. Thus, possibly taking a further subsequence there is an intrinsic flat convergent subsequence:
\be
(X_{j_k}, d_{j_k}, T_{j_k}) \Fto (Y,d_Y,T),
\ee
where either $(Y,d_Y,T)$ is the zero integral current space or  $(Y,d_Y)$ can be viewed as a subspace of $(X, d)$. If $(Y,d_Y, T)$ is the zero integral current space then there is nothing else to prove.

From now on assume that $Y \subset X$. We have to show that $X=Y$.
Recall that the mass measure of $T$, $||T||$, is a finite Borel measure on $\bar Y$ and that by definition of $n$ dimensional integral current space, 
\be 
Y= \left\{ x \in \bar Y \,|\,  \liminf_{r\to 0} \frac{||T||(B_r(x))}{\omega_nr^n}>0  \right\}. 
\ee
To simplify notation we suppose that  $||T||$ is a measure on $X$, where $||T|| (A): = ||T|| (A \cap \bar Y)$ for any Borel set $A \subset X$, and that
\be
(X_j, d_j) \GHto (X,d) \,\,\text{and}\,\, (X_j, d_j, T_j) \Fto (Y,d_Y,T).
\ee

We will first prove that the set of regular points of $X$, $R(X)$, is contained in $Y$.  We claim that the Hausdorff dimension of $X$ is $n$.  Since $(Y, d_Y, T) \neq \bf 0$ is a $n$ dimensional integral current space, 
$\dim_H(Y)=n$.  Hence, $n \leq \dim_H(X)$ and by Burago-Gromov-Perelman's theorem, Theorem \ref{cor-collapsingSeq},  $\dim_H(X) \leq n$.  Thus, we have our claim.

Let  $x \in R(X)$. Then by Lemma \ref{lem-lipmaps}  there exist $x_j \in X_j$, $r_0 > 0$ and bilipschitz functions  $f_j:  B_{r_0}(x_j) \to W_j \subset \R^n$ with $\Lip(f_j),\Lip(f_j^{-1}) \leq L$.  By Sormani's Theorem \ref{thm-SorArzela}, we can assume that
\begin{equation} 
( \bar B_{r_0}(x_j) , d_j, T_j \rstr   \bar B_{ r_0}(x_j) )
\end{equation}
are $n$ dimensional integral currents spaces. Then
\begin{align}\label{eq-pushball}
( f_j(\bar B_{r_0}(x_j)) , d_{\R^n}, f_ {j \#} (T_j \rstr \bar B_{r_0}(x_j))
\end{align}
are $n$ dimensional integral currents spaces.

With no loss of generality, assume that 
$f_j(x_j)=0$ for all $j$. Let $r \in (0,r_0)$. Then recalling that $\Lip(f_j^{-1}) \leq L$ we get $f_j^{-1}(B_{r/L}(0)) \subset B_r(x_j)$. Consider the $n$ dimensional integral current spaces,  
\begin{align} 
M'_j(r)=( B_{r/L}(0) , d_{\R^n}, f_ {j \#} (T_j \rstr  f_j^{-1}(B_{r/L}(0)))).
\end{align}
By Lemma \ref{lem-standcurr},  the spaces given by  (\ref{eq-pushball}) have the standard current structure of $ f_j(\bar B_{r_0}(x_j))  \subset \mathbb R^n$. Hence, the spaces $M'_j(r)$ do not depend on $j$, 
\begin{equation}
M'(r)=M'_j(r)=( B_{r/L}(0) , d_{\R^n},  \Lbrack 1_{B_{r/L}(0)} \Rbrack ).
\end{equation}
Define
\begin{equation} 
M_j(r)=( f_j^{-1}(B_{r/L}(0)), d_j, T_j \rstr f_j^{-1}(B_{r/L}(0))).
\end{equation}

Applying Sormani-Wenger's lemma, Lemma \ref{lem-distFlatLip}, we get the following estimate of the distance between   $M_j(r)$ and $M'(r)$,
\begin{align}\label{eq-ifDist}
d_{\mathcal F} (M_j(r), M'(r) )  \leq  & k_{L,n} \max\{\diam(\spt T_j), \diam(f_j(\spt  T_j))\} \mass( \Lbrack 1_{B_{r/L}(0)} \Rbrack )  \\
\leq &   \frac12 (n+1)L^{n-1} (L-1) 2 r \,  \omega_n (r/L)^n\\
= & \omega_n (n+1)L (L-1) r \, r^n.
\end{align}

By  Theorem \ref{thm-fillvol}, 
\be 
\fillvol ( \bdry M'(r)) =  c (r/L)^n.
\ee

From the previous inequalities and Portegies-Sormani's Theorem \ref{thm-fillvolCont},
\begin{align}\label{eq-fillvol}
\fillvol(\bdry M_j(r)) &\geq    \fillvol  (\bdry M'(r)) - d_{\mathcal F} (M'(r),M_j(r)) \\
&\geq  c (r/L)^n - \omega_n (n+1)L (L-1)  r \, r^n\\
&= \left(  \frac{c}{L^n} - \omega_n (n+1)L (L-1) r    \right) r^n.
\end{align}
For $r$ sufficiently small $C(n,r)=\frac{c}{L^n} - \omega_n (n+1)L (L-1) r $
is positive. Moreover, $\lim_{r \to 0} C(n,r) >0$.

Now, $\mass(T_j \rstr f_j^{-1}(B_{r/L}(0))) \leq \mass(T_j)$ and
$\mass( \bdry (T_j \rstr f _j^{-1}(B_{r/L}(0))))\leq \Lip^{n-1}(f^{-1}_j) \mass(\bdry  \Lbrack 1_{B_{r/L}(0)} \Rbrack )$. 
By Sormani-Wenger's theorem, Theorem \ref{thm-IFinsideGH}, there is a n dimensional integral current space $(B, d, T \rstr B)$ and a subsequence, denoted in the same way, such that 
\begin{equation}\label{eq-IFlimMj}
d_{\mathcal F} \left( M_j(r), (B, d, T \rstr B)\right) \to 0, 
\end{equation}
where either $(B, d, T \rstr B)= \bf 0$ or $B \subset B_r(x)$.  

By the continuity of the filling volume with respect to intrinsic flat distance, Theorem \ref{thm-fillvolCont},  Theorem \ref{thm-fillvol} and  (\ref{eq-ifDist}),
\begin{align}
\fillvol (\bdry  \left((B, d, T \rstr B) \right) = & \lim _{j \to \infty}  \fillvol  \left( \bdry M_j(r)\right) \\
\leq  &  \fillvol  (\bdry M'(r)) + d_{\mathcal F} (M'(r),M_j(r)) <   \infty.  
\end{align}
Thus  $(B, d, T \rstr B) \neq \bf 0$ and  $B \subset B_r(x)$. Once again, by the continuity of the filling volume with respect to intrinsic flat distance and  (\ref{eq-fillvol}),
\begin{align}
||T||(B_r(x)) & \geq  ||T||(B) \geq \fillvol (\bdry  \left((B, d, T \rstr B) \right) \\
& = \lim _{j \to \infty}  \fillvol  \left( \bdry M_j(r)\right) >  C(n, r)r^n.
\end{align}
Thus, 
\be\label{eq-lowerDen}
\liminf_{r\to 0} \frac{||T||(B_r(x))}{\omega_n r^n}> 0.
\ee
This proves that $x \in Y$. Hence, $R(X) \subset Y$.

To complete the proof, by Ambrosio-Kirchheim's characterization of the mass measure, Lemma \ref{lemma-weight} we know that $||T||$ can be written as $||T||=\theta_T \lambda \mathcal{H}^n \rstr Y$ where 
$\theta_T: \bar{Y} \to \N \union \{0\}$ is an integrable function such that
$\theta_T > 0$ in $Y$ and  $\lambda: Y \to \R$ is a nonnegative integrable function bounded below by  $\lambda \geq n^{-n/2}$. By Otsu-Shioya's Theorem \ref{thm-dimSingSet} we know that  $\mathcal H^n (X \setminus R(X))=0$.
It follows that $\mathcal H^n (Y \setminus R(X))=0$. Thus, for $x \in X$
\begin{align}
||T||(B_r(x)) &=  \int_{B_r(x)}\theta_T(y)\lambda(y)d\mathcal H^n\rstr Y \\
&=  \int_{B_r(x) \cap R(X)} \theta_T(y)\lambda(y)d \mathcal H^n\rstr Y \\
&\geq n^{-n/2} \mathcal{H}^n(B_r(x)) \\
& \geq  n^{-n/2} C(\kappa,n, D) r^n   
\end{align}
where $C(\kappa,n, D)=\mathcal{H}^n(X)/D^n$  when $\kappa \geq 0$ by Burago-Gromov-Perelman's Bishop-Gromov volume comparison for Alexandrov spaces, Theorem \ref{thm-BishopGromov}. Otherwise,  $C(\kappa,n, D)=\omega_n$ by Bishop volume comparison. It follows that $\liminf_{r\to 0} ||T||(B_r(x))/r^n>0$. This shows that $X \subset Y$.
\end{proof}


\appendix

\section{Filling Volume of  Spheres in Euclidean Space\\*[2ex] }\label{app}

Since the mass measure of an integral current space is lower semicontinuous with respect to intrinsic flat distance \cite{CS11},  here we study the notion of filling volume of a current. See \cite{CS10}.  In Theorem \ref{thm-fillvol} we see that the filling volume of a $n$ dimensional sphere of radius $r$ in Euclidean space rescales as $r^n$ times the filling volume of the $n$ dimensional sphere in Euclidean space of radius $1$.

\begin{defn}(c.f. \cite{PorSor})\label{defn-fillvol} 
Given an $n$ integral current space $M=(X,d,T)$, $n \geq 1$, define  the filling volume of $\bdry M= (\set(\bdry T),d,\bdry T)\neq \bf{0}$
by 
\be\label{eq-fillvol}
\fillvol(\bdry M)= \inf \{ \mass(N) \, |\, N=(Y,d_Y,U) \in \mathcal{M}^n, \,\, \bdry M= \bdry N\},
\ee
where $\bdry M= \bdry N$ means that there is an orientation preserving isometry between  $\bdry M$ and $\bdry N$. That is, 
there exists an isometry $\varphi: \set(\bdry T) \to  \set (\bdry U)$ such that $\varphi_\sharp \bdry T = \bdry U$.  
For $\bdry M= \bf{0}$ we set $\fillvol(\bdry M)=\infty.$
\end{defn}

\begin{rmrk}\label{rmrk-fillvolReal}
The infimum in (\ref{defn-fillvol}) is achieved when the space is precompact, see Remark 2.47 in Portegies-Sormani \cite{PorSor}. 
Moreover, it equals a positive number. 
\end{rmrk}

The notion of filling volume given in Definition \ref{eq-fillvol} is not exactly the same notion introduced by Gromov \cite{Gromov-filling}, however many similar properties hold. Gromov's  filling volume is defined using chains rather than integral current spaces and the notion of volume used by Gromov is not the same as Ambrosio-Kirchheim's mass.   

From the definition of filling volume and mass it follows that for  $M=(X,d,T) \in  \mathcal{M}^n$,
\be 
||T||(X) = \mass(T) \geq \fillvol(\bdry M). 
\ee

The continuity of the filling volume with respect to the intrinsic flat convergence follows from the following theorem. This fact was first observed by Sormani-Wenger \cite{CS10} building upon work by Wenger on flat convergence of integral currents in metric spaces \cite{Wenger-flat}.  

\begin{thm}[cf. Portegies-Sormani \cite{PorSor}]\label{thm-fillvolCont}
For any pair of integral current spaces, $M_i=(X_i,d_i,T_i)$, 
\be
\fillvol(\bdry M_1) \leq \fillvol(\bdry M_2) + d_{\mathcal F}(M_1,M_2).
\ee
\end{thm}

\begin{thm}\label{thm-fillvol}
Let $r>0$ and consider the $n$ dimensional integral current space  
\be M'(r)=( B_{r}(0) \subset \mathbb R^n, d_{\R^n},  \Lbrack 1_{B_{r}(0)} \Rbrack ).
\ee
Then, 
\be 
\fillvol (\bdry M'(r))  =  \fillvol (\bdry M'(1))   r^n.
\ee
\end{thm}

%
%
%
%

\begin{proof}
Let $s> 0$. By Remark \ref{rmrk-fillvolReal} we know that there is a $n$ integral current space  $N= (Z, d, U)$ where $\fillvol(\bdry M'(s))$ is realized. 
That is, there exists an orientation preserving isometry 
\be 
\varphi: \set (\bdry   \Lbrack 1_{B_{s}(0)} \Rbrack )  \to \set (\bdry U)
\ee
such that 
\be 
\varphi_\#  \bdry  \Lbrack 1_{B_{s}(0)} \Rbrack=  \bdry U
\ee
and 
 \be 
\fillvol (\bdry M'(s)) = \mass(U).
\ee

Note that for any $t>0$,  $\set (\bdry  \Lbrack 1_{B_{t}(0)} \Rbrack)= \bdry B_t(0)$.  Define $R: \bdry B_{r}(0) \to \bdry B_{s}(0)$ by $R(x)= \tfrac{s}{r}x$ and  let $\iota: (Z,d) \to (Z,\tfrac{r}{s}d)$ denote the identity map.  We claim that 
\be 
\tilde N=(Z, \tfrac{r}{s}d,  \iota_{\# }  U)
\ee
is a $n$ integral current space and that 
\be
\tilde \varphi= \iota \circ \varphi \circ R: \bdry B_r(0) \to  \iota (\set (\bdry U))
\ee
is an orientation preserving isometry. 

The pushforward of a current is a current and, from the definition of the mass measure it follows that 
\be 
\| \iota_\#   U \| \leq \Lip(\iota)^n \iota_\# \| U\|.
\ee
 Since $\iota$ is invertible and $\Lip(\iota^{-1})= \Lip^{-1}(\iota)$, we get 
 \be\label{eq-massiotaU}
 \| \iota_\#   U \| = \Lip(\iota)^n \iota_\# \| U\|. 
 \ee

This shows that 
\be
\set (\iota_\#   U)= \iota (\set(U)).
\ee
The pushforward and the boundary operator commute, hence $\bdry ( \iota_\#   U)= \iota_\#  \bdry U$ is a current.  Thus, $\tilde N $ is a $n$ integral current space.  With the same reasoning as before we see that 
\be 
\| \iota_\# \bdry  U \| = \Lip(\iota)^{n-1} \iota_\# \| \bdry U\|
\ee
and 
\be
\set (\iota_\#   \bdry U)= \iota(\set( \bdry U)).
\ee
Hence, $\set ( \bdry  \iota_\#  ( U)= \iota(\set( \bdry U))$.

By definition of pushforward, the definition of $R$ and change of variables in $\mathbb R^n$,
\be 
R_\#  \bdry  \Lbrack 1_{B_{r}(0)} \Rbrack =  \bdry  \Lbrack 1_{B_{s}(0)} \Rbrack.
\ee
Recalling that $N= (Z, d, U)$ realizes $\fillvol(\bdry M'(s))$, 
\be
 \tilde \varphi_\#  \bdry   \Lbrack 1_{B_{r}(0)} \Rbrack =  \iota_\#  ( \bdry U ) = \bdry ( \iota_\#   U ). 
\ee
This proves that $\tilde \varphi: \bdry M'(r) \to \bdry \tilde N$ is an orientation preserving isometry. 

Then by (\ref{eq-massiotaU}), given that $\Lip(\iota)=\tfrac{r}{s}$ and noting that $\iota$ equals the identity function,
 \be
 \fillvol (\bdry M'(r)) \leq  \mass(  \iota_\#   U) =( \tfrac {r}{s} )^n \mass(U).
\ee
Applying this inequality twice, we conclude that 
 \be 
 \fillvol (\bdry M'(r)) = ( \tfrac {r}{s} )^n  \fillvol (\bdry M'(s)).
\ee
Hence, 
\be 
 \fillvol (\bdry M'(r)) = r^n  \fillvol (\bdry M'(1))  .
\ee
\end{proof}

%

\vskip 5mm

\bibliographystyle{amsalpha}

\begin{thebibliography}{A}



\bibitem {AlexanderBishop} Stephanie Alexander and Richard Bishop,
  \textit{FK convex functions on metric spaces},
  Manuscripta Math, 110
  \textbf{} (2003) no. 1, 115-133.



\bibitem {AKP} Stephanie Alexander, Vitali Kapovitch and Anton Petrunin,
  \textit{Alexandrov geometry},
  a draft avaliable at {\it www.math.psu.edu/petrunin}.


\bibitem {AK} Luigi  Ambrosio and Bernd Kirchheim,
  \textit{Currents in metric spaces},
  Acta Mathematica, 185
  \textbf{} (2000) no. 1.


\bibitem {BBI} Dmitri Burago, Yuri Burago and  Sergei Ivanov,
  \textit{A Course in Metric Geometry},
  Graduate studies in mathematics, AMS, 33, (2001).



\bibitem {BGP} Yuri Burago, Mikhael Gromov and Grigori Perel'man,
  \textit{A.D. Alexandov spaces with curvature bounded below},
  Uspekhi Mat. Nauk, 47:2, (1992), 3--51;
  translation in Russian Math. Surveys,
  47:2, (1992), 1--58.


\bibitem{FF} Herbert Federer and Wendell H. Fleming,
\textit{Normal and integral currents}, 
Ann. of Math. (2), Vol 72 (1960), 458--520.






\bibitem{Gro81}  Mikhael Gromov,
\textit{Structures m\'etriques pour les vari\'et\'es riemanniennes},
Textes Math\'ematiques, Vol. 1,
Edited by J. Lafontaine and P. Pansu,
CEDIC, Paris, 1981. 



\bibitem {Gromov} Mikhael Gromov,
  \textit{Groups of Polynomial growth and expanding maps},
  Institute Hautes Etudes Sci. Publ. Math., 53, (1981), 53--73.
  
  
\bibitem{Gromov-filling} Mikhael Gromov,
\textit{Filling {R}iemannian manifolds},
J. Differential Geom.,  Vol 18, Issue 1, (1983), 1--147.
  
  \bibitem {Gro13} Mikhael Gromov,
  \textit{Dirac and Plateau Billiards in Domains with Corners},
  Central European Journal of Mathematics, Vol 12, Issue 8, (2014), 1109--1156.


\bibitem {Gro14} Mikhael Gromov,
\textit{Plateau Stein Manifolds},
Central European Journal of Mathematics, Vol 12, Issue 7, (2014), 923--951.


\bibitem{Ho} Shouhei Honda,
\textit{Ricci curvature and Orientability},
Calc. Var. 56: 174, (2017).



 \bibitem{Ja} Maree Jaramillo, Raquel Perales, Priyanka Rajan, Catherine Searle and Anna Siffert,
\textit{Alexandrov Spaces with Integral Current Structure},
arXiv:1703.08195, math.DG. To appear in 


 
\bibitem {LakzianDiam} Sajjad Lakzian,
  \textit{On diameter controls and smooth convergence away from singularities},
  DGA, Vol 47
  \textbf{} (2016) 99--129.


\bibitem {LR10} Nan Li and Xiaochun Rong,
{\it Relatively maximum volume rigidity in Alexandrov geometry}, Pacific J. of Mathematics, 259 no. 2, (2012) 387--420.

\bibitem{MP} Rostislav Matveev and Jacobus Portegies
 {\em Intrinsic flat and Gromov-Hausdorff convergence of manifolds with Ricci curvature bounded below}, 
 Journal of Geometric Analysis, Vol 27, 3, (2017) 1855-1873.
 

\bibitem{M2}  Mitsuishi, A. {\em Orientability and fundamental classes of Alexandrov spaces with applications},  arXiv:1610.08024, (2016), math.MG.


 \bibitem {MunnInt} Michael Munn,
  \textit{Intrinsic flat convergence with bounded Ricci curvature},
  arXiv: 1405.3313v2
  \textbf{} (2015) math.MG.
  
\bibitem {NuZiPe} Jesus Nunez-Zimbron and Raquel Perales,
\textit{A generalized tetrahedral property for spaces with conical singularities},
  arXiv: 1709.05877
   \textbf{} (2017) math.DG.


\bibitem {OS} Yukio Otsu and Takashi Shioya,
  \textit{The riemannian structure
of Alexandrov spaces},
  J. Differential Geometry, 39, (1994), 629--658.

\bibitem{Perales-1} Raquel Perales,
\textit{Volumes and Limits of Manifolds with Ricci Curvature and
Mean Curvature Bounds}, 
Diff. Geometry and its Applications,  Vol 48, (2016) 23--37.


\bibitem{Perales-2} Raquel Perales,
\textit{Convergence of Manifolds and Metric Spaces with Boundary}, 
arXiv:1505.01792, (2015) math.MG.


\bibitem {PortegiesSemi} Jacobus W. Portegies,
\textit{Semicontinuity of eigenvalues under intrinsic flat convergence},
 Calc. of Var. and PDE, Vol 54 (2), (2015) 1725--1766.
  
  
\bibitem{PorSor} Jacobus Portegies and Christina Sormani, 
\textit{Properties of the Intrinsic Flat Distance},
Algebra i Analiz Issue 3, Volume 29 (2017),  70--143.
  

\bibitem {SorArzela} Christina Sormani,
\textit{Intrinsic Flat Arzela-Ascoli theorems},
Communications in Analysis and Geometry Vol. 27, No 1, 2019.


\bibitem {CS10} Christina Sormani and Stefan Wenger,
  \textit{Weak convergence of currents and cancellation},
  Calc. Var. Partial Differential Equations, 38 No. 1-2, (2010) 183-206.


\bibitem {CS11} Christina Sormani and Stefan Wenger,
  \textit{The intrinsic flat distance between Riemannian manifolds and other integral current spaces},
  J. Differential Geom., 87  no. 1, (2011) 117-199.



\bibitem {Wenger-flat} Stefan Wenger,
\textit{Flat convergence for integral currents in metric spaces},
Calc. Var. Partial Differential Equations, Vol 28, No 2, (2007), 139--160.


\bibitem {We11} Stefan Wenger,
  \textit{Compactness for manifolds and integral currents with bounded diameter and volume},
  Calc. Var. Partial Differential Equations, 40 Issue 3-4,  (2011) 423--448.
  
  
\bibitem{Whit} Hassler Whitney,
\textit{Geometric integration theory},
Princeton University Press, 1957. 



\end{thebibliography}


\end{document}